\documentclass[12pt]{article}

\usepackage{amsmath,amssymb,amsthm,amsfonts}
\usepackage{color}

\newcommand{\Real}{\mathbb R}

\newcommand{\abs}[1]{\left\vert#1\right\vert}

\renewcommand{\phi}{\varphi}
\newcommand{\eps}{\varepsilon}
\renewcommand{\ge}{\geqslant}
\renewcommand{\le}{\leqslant}

\newtheorem{thm}{Theorem}
\newtheorem{prop}{Proposition}
\newtheorem{lm}{Lemma}

\newcommand{\Wf}{\stackrel{o\ }{W{}_1^1}}

\newcommand{\I}{\mathcal{I}}
\newcommand{\J}{\mathcal{J}}
\newcommand{\K}{\mathcal{K}}
\newcommand{\B}{{\bf b}}

\newcommand{\Ln}{\mathop{\rm ln}}

\title{On monotonicity of some functionals with variable exponent under symmetrization}
\author{
S.~Bankevich \footnote{St. Petersburg State University, Russia;
Sergey.Bankevich@gmail.com}
\and A.~I.~Nazarov \footnote{St. Petersburg Dept of Steklov Math
Institute and St. Petersburg State University, Russia;
al.il.nazarov@gmail.com} }
\date{}

\begin{document}
\maketitle


{\small
\section*{Abstract}

We give necessary and sufficient conditions for the P\'olya--Szeg\"o type inequality with variable exponent of summability.

MSC classes: 49K05, 49N99

}

\section{Introduction}

First, we recall the layer cake representation for a measurable function $u: [-1, 1] \to \Real_+$
(here and elsewhere $\Real_+ = [0,\infty)$).
Namely, if we set $\mathcal{A}_t: = \{x \in [-1,1]:\ u(x)> t \}$
then $u(x) = \int_0^\infty \chi_{\mathcal{A}_t} \, dt$.

We define the symmetric rearrangement (symmetrization) of a measurable set $E \subset [-1, 1]$ and the
symmetrization of a nonnegative function $u \in \Wf(-1, 1)$ as follows:
\begin{eqnarray*}
E^*: = [-\frac{\abs{E}}{2}, \frac{\abs{E}}{2}]; \qquad
u^*(x): = \int\limits_0^\infty \chi_{\mathcal{A}_t^*} \, dt.
\end{eqnarray*}

The well-known P\'olya-Szeg\"o inequality for $u \in \Wf(-1, 1)$ reads
\begin{equation}
\label{PS}
I(u^*) \le I(u), \qquad \text{where } I(u) = \int\limits_{-1}^1 |u'(x)|^p dx, \quad p \ge 1
\end{equation}
(in particular $I(u^*)=\infty$ implies $I(u)=\infty$).

The generalizations of the inequality (\ref{PS}) were considered in a number of papers, see, e.g., the survey \cite{Tal} and references therein.
In particular, the paper \cite{Br} deals with functionals of the form
\begin{equation*}
\label{functional}
I(\mathfrak a, u) = \int\limits_{-1}^1 F\big(u(x), \mathfrak a(x, u(x)) \abs{u'(x)}\big) \, dx
\end{equation*}
and their multidimensional analogues. Here $\mathfrak a: [-1, 1] \times \Real_+ \to \Real_+$ is a continuous function convex w.r.t. the first argument
while $F: \Real_+ \times \Real_+ \to \Real_+$ is a continuous function convex w.r.t. the second argument, $F( \cdot, 0 ) \equiv 0$.
A gap in the proof in \cite{Br} was filled in
\cite{BN-CV} where the inequality $I(\mathfrak a, u^*)\le I(\mathfrak a, u)$ was proved for the natural class of functions $u$. Similar results
for monotone rearrangement were also obtained in \cite{BN-DAN, BN-CV, B-ZNS}.

In this paper we consider functionals with variable summability exponent
\begin{equation*}
\J(u) = \int\limits_{-1}^1 |u'(x)|^{p(x)} dx, \qquad \I(u) = \int\limits_{-1}^1 ( 1 + | u'(x) |^2 )^{\frac {p(x)}{2}} dx.
\end{equation*}
Here $p(x) \ge 1$ is a continuous function defined on $[-1, 1]$, $u \in \Wf[-1, 1]$, $u \ge 0$. Such functionals arise in some problems of mathematical physics,
in particular in modeling of electrorheological fluids, see, e.g., \cite{Zh} and \cite{DHHR}.

Our results were partially announced in \cite{B-FAA}.

The article is divided into 6 sections. In Section 2 we deduce the necessary conditions for the P\'olya-Szeg\"o type inequalities.
It turns out that for the functional $\J$ such inequality holds only in trivial case $p\equiv const$ while for the functional $\I$
we have non-trivial conditions. In Section 3 we show that necessary conditions are also sufficient for the inequality $\I(u^*) \le \I(u)$.
In Section 4 we give some explicit sufficient conditions for the exponent $p$ which ensure this inequality.

In Section 5 we consider a multidimensional analogue of the inequality $\I(u^*) \le \I(u)$. Unexpectedly, such inequality can hold only for
the exponent $p$ which is constant w.r.t. the variable of rearrangement.

Section 6 (Appendix) contains a numerical-analytical proof of some inequalities in Section 4.

\section{Necessary conditions}

\begin{thm}
\label{uniform}
Suppose $\J(u^*) \le \J(u)$ holds for any piecewise linear function $u\ge0$.
Then $p(x) \equiv const$.
\end{thm}

\begin{proof}
Consider a point $x_0 \in (-1, 1)$. For every $\alpha > 0$ and $\eps > 0$ such that $[x_0 - \eps, x_0 + \eps] \subset [-1, 1]$ we define
$$
u_{\alpha,\eps}(x) = \alpha ( \eps - |x - x_0| )_+.
$$
Then $u_{\alpha,\eps}^*(x) = \alpha ( \eps - |x| )_+$, and
$$
\J(u_{\alpha, \eps}) = \int\limits_{x_0 - \eps}^{x_0 + \eps} \alpha^{p(x)} dx, \qquad
\J(u_{\alpha, \eps}^*) = \int\limits_{-\eps}^{\eps} \alpha^{p(x)} dx.
$$
We take the inequality
$$
\frac {\J(u_{\alpha, \eps}^*)}{2 \eps} \le \frac {\J(u_{\alpha, \eps})} {2 \eps},
$$
and push $\eps$ to zero. Since $p$ is continuous, this gives $\alpha^{p(0)} \le \alpha^{p(x_0)}$.
Taking $\alpha > 1$ and $\alpha < 1$ we arrive at $p(0) \le p(x_0)$ and $p(0) \ge p(x_0)$ respectively.
\end{proof}


\begin{thm}
\label{necessary}
Suppose inequality $\I(u^*) \le \I(u)$ holds for any piecewise linear function $u\ge0$.
Then $p$ is even and convex.
Moreover, the following function is convex:
$$K(s, x) = s ( 1 + s^{-2} )^{\frac {p(x)}{2}} \qquad s > 0,\ x \in [-1, 1].$$
\end{thm}

To prove Theorem \ref{necessary} we need the following
\begin{prop}
\label{convProp}
{\rm \cite[Lemma 10]{BN-CV}.}
Let $p$ be a continuous function on $[-1, 1]$.
Suppose that
$$
\forall s, t \in [-1, 1] \qquad p(s) + p(t) \ge p\big(\frac {s - t}{2}\big) + p\big(\frac{t - s}{2}\big).
$$
Then $p$ is even and convex.
\end{prop}

\begin{proof}[Proof of Theorem \ref{necessary}]
Fix two points $-1 < x_1 < x_2 < 1$
and consider a finite piecewise linear function with nonzero derivative only in circumferences of $x_1$ and $x_2$.
Namely, for arbitrary $s, t > 0$ and a small enough $\eps > 0$, we set
$$
u_\eps(x) = \min\big( 2 \eps, ( \eps + s^{-1} (x - x_1))_+, ( \eps + t^{-1} (x_2 - x) )_+\big).
$$
Then
$$
u_\eps^*(x) = \min\big( 2 \eps, ( \eps + ( t + s )^{-1} ( x_2 - x_1 - 2 |x| ) )_+ \big).
$$

The sets where $u_\eps' = 0$ and $u_\eps^*{}' = 0$ have equal measures.
Therefore, the inequality $\I(u_\eps^*) \le \I(u_\eps)$ can be rewritten as follows:
\begin{multline*}
\int\limits_{ x_1 - s \eps }^{ x_1 + s \eps } \big( 1 + \frac 1 {s^2 } \big)^{\frac {p(x)} 2} dx
+ \int\limits_{ x_2 - t \eps }^{ x_2 + t \eps } \big( 1 + \frac 1 {t^2 } \big)^{\frac {p(x)} 2} dx \ge
\\ \ge \int\limits_{ \frac {x_1 - x_2} 2  - \frac {s + t}2 \eps }^{ { \frac {x_1 - x_2} 2 } + { \frac {s + t} 2 } \eps }
\Big( 1 + \frac{1}{ ( {\frac { s + t} 2 } )^2 } \Big)^{\frac {p(x)} 2} dx
     + \int\limits_{ {\frac {x_2 - x_1} 2 } - { \frac {s + t} 2 } \eps }^{ { \frac {x_2 - x_1} 2 } + { \frac {s + t} 2 } \eps }
     \Big( 1 + \frac{1}{ ( { \frac {s + t} 2 } )^2 } \Big)^{\frac {p(x)} 2} dx.
\end{multline*}

We divide this inequality by $2 \eps$ and push $\eps \to 0$.
This gives
\begin{multline}
\label{preConv}
s \Big( 1 + { \frac 1 {s^2} } \Big)^{\frac {p(x_1)} 2} + t \Big( 1 + { \frac 1 {t^2} } \Big)^{\frac {p(x_2)} 2}\\
\ge { \frac {s + t} 2 } \Big( 1 + \frac{1}{ ( { \frac {s + t} 2 } )^2 } \Big)^{\frac {p( { \frac {x_1 - x_2} 2 } )} 2}
  + { \frac {s + t} 2 } \Big( 1 + \frac{1}{ ( { \frac {s + t} 2 } )^2 } \Big)^{\frac {p( { \frac {x_2 - x_1} 2 } )} 2}.
\end{multline}

First, we put $s = t$ in (\ref{preConv}). This gives
\begin{equation}
\label{s_eq_t}
( 1 + { \frac 1 {s^2} } )^{\frac {p(x_1)} 2} + ( 1 + { \frac 1 {s^2} } )^{\frac {p(x_2)} 2}
\ge ( 1 + { \frac 1 {s^2} } )^{\frac {p( { \frac {x_2 - x_1} 2 } )}  2} + ( 1 + { \frac 1 {s^2} } )^{\frac {p( { \frac {x_1 - x_2}2 } )} 2}.
\end{equation}

We define $\sigma: = {\frac 1 {s^2}}$, apply Taylor's expansion at $\sigma = 0$ in (\ref{s_eq_t}) and arrive at
$$
\sigma p(x_1) + \sigma p(x_2) \ge \sigma p( {\frac { x_2 - x_1}2 } ) + \sigma p( {\frac { x_1 - x_2} 2 } ) + r(\sigma),
$$
where $r(\sigma) = o(\sigma)$ as $\sigma \to 0$.
Hence for any $x_1, x_2 \in [-1, 1]$ we have
$$
p(x_1) + p(x_2) \ge p( { \frac {x_2 - x_1} 2 } ) + p( { \frac {x_1 - x_2} 2 } ).
$$
By Proposition \ref{convProp}, $p$ is even and convex.

Finally we put $-x_2$ instead of $x_2$ in (\ref{preConv}). Since $p$ is even, we obtain
$K(s, x_1) + K(t, x_2) \ge 2 K( { \frac {s + t}2 }, {\frac { x_1 + x_2}2 } )$.
\end{proof}

\section{Proof of inequality $\I(u^*) \le \I(u)$}

In this section we show that necessary conditions from Theorem \ref{necessary} are sufficient as well.

\begin{lm}
\label{quasiConv}
Let $m$ be an even positive number, let $s_k > 0$ ($k = 1 \dots m$), and let $-1 \le x_1 \le \dots \le x_m \le 1$.
Suppose that $K(s, x)$ is even in $x$ and jointly convex in $s$ and $x$.
Then
\begin{equation}
\label{K_ineq}
\sum\limits_{k = 1}^{m} K(s_k, x_k) \ge
2 K\Big({\frac 1 2} \sum\limits_{k = 1}^{m} s_k, {\frac 1 2} \sum\limits_{k = 1}^{m} (-1)^k x_k\Big).
\end{equation}
\end{lm}

\begin{proof}
Note that inequality (\ref{K_ineq}) is equivalent to the same inequality for the function $M(s, x) = K(s, x) - s$.
Also note that $M$ decreases in $s$ since $M$ is convex in $s$ and
$$
M_s(s, x) = (1 + {\frac 1 {s^2}})^{{\frac {p(x)} 2} - 1} (1 + {\frac 1 {s^2}} - {\frac {p(x)}{ s^2}}) - 1 \rightarrow 0 \qquad {\rm as} \quad s \to \infty.
$$
Therefore,
\begin{multline*}
\sum\limits_{k = 1}^{m} M(s_k, x_k)
\ge M(s_1, x_1) + M(s_m, x_m)
\overset{a}{\ge} 2 M\big({\frac {s_1 + s_m} 2}, {\frac {x_m - x_1} 2}\big) \ge \\
\overset{b}{\ge} 2 M\Big({\frac 1 2} \sum\limits_{k = 1}^{m} s_k, {\frac {x_m - x_1} 2}\Big)
\overset{c}{\ge} 2 M\Big({\frac 1 2} \sum\limits_{k = 1}^{m} s_k, {\frac 1 2} \sum\limits_{k = 1}^{m} (-1)^k x_k\Big).
\end{multline*}
Here (a) follows from evenness of $M$ in $x$ and its convexity,
(b) follows from decreasing of $M$ in $s$,
(c) follows from increasing of $M$ in $x \ge 0$.
\end{proof}

\begin{lm}
\label{linear}
Suppose that $K(s, x)$ is even in $x$ and jointly convex in $s$ and $x$.
Then $\I(u^*) \le \I(u)$ for any piecewise linear nonnegative function $u \in \Wf[-1, 1]$.
\end{lm}

\begin{proof}
Denote by $L \subset [-1, 1]$ the set of nodes of $u$ (including the endpoints of the segment).
Take $U = u([- 1, 1]) \setminus u(L) $, i.e. $U$ is the image of $u$ without images of the nodes.
This set is a union of a finite number of disjoint intervals $U = \cup_j U_j$.
Note that for each $j$ the set $u^{-1}(U_j)$ is a union of an even number (say, $m_j$) of disjoint intervals.
Moreover, $u$ coincides with some linear function $y^j_k$, $k = 1, \dots, m_j$, on each interval.
Without loss of generality, we assume that the supports of $y^j_k$ are ordered by $k$ for any $j$,
that is $\sup dom(y^j_k) \le \inf dom(y^j_{k + 1})$.
Denote $b^j_k = |y^j_k{}'(x)|$ and
$$
Z = {\rm meas}\{ x \in (-1, 1) | u'(x) = 0 \} = {\rm meas}\{ x \in (-1, 1) | u^*{}'(x) = 0 \}.
$$
Then
\begin{multline*}
\I(u) - Z = \sum\limits_j \int\limits_{u^{-1}(U_j)} (1 + u'^2(x))^{\frac {p(x)} 2} dx
= \sum\limits_j \sum\limits_k \int\limits_{dom(y^j_k)} (1 + {y^j_k}'^2(x))^{\frac {p(x)} 2} dx =
\\ = \sum\limits_j \int\limits_{U_j} \sum\limits_k {\frac 1 {b^j_k}} (1 + b^j_k{}^2)^{\frac {p((y^j_k)^{-1}(y))} 2} dy
= \sum\limits_j \int\limits_{U_j} \sum\limits_k K\Big({\frac 1 {b^j_k}}, (y^j_k)^{-1}(y)\Big) dy.
\end{multline*}

Any point $y \in U$ has two symmetrical preimages with respect to function $u^*$,
therefore we can define $(u^*)^{-1}: U \to [0, 1]$.
Thus, for each $j$ we have
\begin{eqnarray*}
(u^*)^{-1} (y) & = & {\frac 1 2} \sum\limits_{k = 1}^{m_j} (-1)^k (y^j_k)^{-1}(y); \\
|((u^*)^{-1})'(y)| & = & {\frac 1 {|u^*{}'((u^*)^{-1}(y))|}} = {\frac 1 2} \sum\limits_{k = 1}^{m_j} {\frac 1 {b^j_k}} =: { \frac 1 {b_j^*} }.
\end{eqnarray*}
Since $u^*$ is even, we obtain
\begin{multline*}
\I(u^*) - Z = 2 \int\limits_{(u^*)^{-1}(U)} (1 + u^*{}'^2(x))^{p(x) \over 2} dx =
\\ = 2 \int\limits_U |((u^*)^{-1})'(y)| \cdot \big( 1 + { 1 \over ((u^*)^{-1})'(y)^2 } \big)^{p((u^*)^{-1}(y)) \over 2} dy =
\\ = 2 \sum\limits_j \int\limits_{U_j} {1 \over b_j^*} \big(1 + b_j^*{}^2 \big)^{{1 \over 2} p\big({1 \over 2} \sum\limits_{k = 1}^{m_j} (-1)^k (y^j_k)^{-1}(y)\big)} dy =
\\ = 2 \sum\limits_j \int\limits_{U_j} K\Big({1 \over 2} \sum\limits_{k = 1}^{m_j} {1 \over b^j_k}, {1 \over 2} \sum\limits_{k = 1}^{m_j} (-1)^k (y^j_k)^{-1}(y)\Big) dy.
\end{multline*}
To complete the proof, it is sufficient to show the following inequality for each $j$ and $y \in U_j$:
$$
\sum\limits_{k = 1}^{m_j} K\Big({1 \over b^j_k}, (y^j_k)^{-1}(y)\Big) \ge
2 K\Big({1 \over 2} \sum\limits_{k = 1}^{m_j} {1 \over b^j_k}, {1 \over 2} \sum\limits_{k = 1}^{m_j} (-1)^k (y^j_k)^{-1}(y)\Big).
$$
But this inequality is provided by Lemma \ref{quasiConv}.
\end{proof}

Now we can prove the result in general case.

\begin{thm}
\label{mainThm}
Suppose that $p$ is even and $K$ is jointly convex in $s$ and $x$.
Then for any nonnegative function $u \in \Wf[-1, 1]$ the inequality $\I(u^*) \le \I(u)$ holds.
\end{thm}

\begin{proof}
Without loss of generality, we can assume $I(u) < \infty$.
Since $p(x)$ is bounded, we can choose a sequence of piecewise constant functions $v_n$ converging to $u'$ in Orlicz space $L^{p(x)}[-1, 1]$,
see, e.g., \cite[Theorem~1.4.1]{Shar}. Denote by $u_n$ the primitive of $v_n$.
Changing $v_n$ if necessary we can assume without loss of generality that $u_n \ge 0$ and $u_n(\pm 1) = 0$.

By embedding $L^{p(x)}[-1, 1]\mapsto L^1[-1, 1]$ we have $u_n \to u$ in $\Wf[-1, 1]$.
Further, since $| \sqrt{ 1 + x^2 } - \sqrt{ 1 + y^2 } | \le | x - y |$ for all $x$ and $y$,
the relation $v_n \to u'$ in $L^{p(x)}$ implies $\I( u_n ) \to \I( u )$.

%

By \cite[Theorem 1]{Br}, the convergence $u_n \to u$ in $\Wf[-1, 1]$ implies weak convergence $u_n^* \rightharpoondown u^*$ in $\Wf[-1, 1]$.
By the Tonelli theorem, the functional $\I$ is sequentially weakly lower semicontinuous (see, e.g., \cite[Theorem 3.5]{BGH}).
Thus,
$$
\I(u^*) \le \liminf\limits_n \I(u_n^*) \le \lim\limits_n \I(u_n) = \I(u).
$$
\end{proof}

\section{On some sufficient conditions}
The condition of joint convexity of function $K$ is in fact certain assumption on function $p$. It is easy to see that
$\partial^2_{ss}K>0$, and $\partial^2_{xx}K\ge0$ since $p$ is convex. Thus, the convexity of $K$ is equivalent to $\det(K'')\ge0$ in the sense of distributions.
Direct calculation gives
\begin{multline*}
\det(K'') = \frac{(1 + w)^{q - 1}}{4}\\
\times\Big( w (w q + 1) (q + 1) \Ln(1 + w) (q'^2 \Ln(1 + w) + 2 q'') - q'^2 ( (1 - q w) \Ln(1 + w) - 2 w )^2 \Big),
\end{multline*}
where $q = q(x) = p(x) - 1$ and $w = w(s) = \frac{1}{s^2}$.

Thus, we arrive at the inequality for the function $q$
\begin{equation}
\label{qq''} q q'' \ge q'^2 {\cal A}(q)\equiv
q'^2\cdot\sup\limits_{w>0} A(w,q),
\end{equation}
where
$$
A(w, q) = \frac{
q ( 4 w - ( w + 3 ) \ln( w + 1 ) ) - {w - 1 \over w} \ln( w + 1 ) + 4 {w \over \ln( w + 1 )} - 4
}{
2 ( q w + 1 )
} \cdot {q \over q + 1}.
$$

The following statement is evident.

\begin{lm}
\label{conv}
Let $q\ge0$ is a continuous function on $[-1, 1]$. Then the inequality $q q'' \ge q'^2 \mathcal M$ in the sense of distributions with $\mathcal M\in(0,1)$ is equivalent to convexity
of the function $q^{1-\mathcal M}$.
\end{lm}

Now we can give a simple sufficient condition for the monotonicity of the functional $\cal I$ under symmetrization.

\begin{thm}
 Let $p(x)\ge1$ be an even continuous function on $[-1, 1]$.

 {\bf 1}. If the function $(p(x)-1)^{0.37}$ is convex then the inequality $\I(u^*) \le \I(u)$ holds for any nonnegative function $u\in\Wf[-1,1]$.

 {\bf 2}. If $p(x)\le 2.36$ for $x\in[-1,1]$ and the function $\sqrt{p(x)-1}$ is convex then the inequality $\I(u^*) \le \I(u)$ holds for any nonnegative function $u\in\Wf[-1,1]$.
 \end{thm}

\begin{proof}
The following inequalities are proved in Appendix:
\begin{eqnarray}
\label{calc}
&& \sup\limits_{q\ge0}{\cal A}(q)=\limsup\limits_{q\to+\infty}{\cal A}(q)\le0.63;\\
\label{calcHalf}
&& \sup\limits_{0\le q\le1.36}{\cal A}(q)\le0.5.
\end{eqnarray}
By Lemma \ref{conv} the inequality (\ref{qq''}) holds in both cases {\bf 1} and {\bf 2}. Theorem \ref{mainThm} completes the proof.
\end{proof}

\section{Multidimensional case}

In this section we assume that $\Omega=\omega\times (-1,1)$ is a cylindrical domain in $\Real^n$, and $x=(x',y)$ where $x'\in\omega$, $y\in(-1,1)$.
For $u\in\Wf(\Omega)$ we denote by $u^*$ the Steiner symmetrization of $u$ along $y$ that is one-dimensional symmetrization w.r.t. $y$ for every $x'$.

We introduce a multidimensional analogue of the functional $\I$:
$$
\widehat\I(u)=\int\limits_{\Omega} ( 1 + | \nabla u(x) |^2 )^{\frac {p(x)}{2}} dx.
$$


\begin{thm}
Suppose that $\widehat\I(u^*) \le \widehat\I(u)$ for any nonnegative function $u \in \Wf(\Omega)$.
Then $p(x',y)$ is independent on $y$.
\end{thm}

\begin{proof}
First of all, we claim, similar to Theorem \ref{necessary}, that
$p$ should be even and convex in $y$, and the function
$$
\K_{x'}(c, d, y) = c \big({\textstyle1 + {1 + d^2 \over c^2}}
\big)^{p(x', y) \over 2}
$$
should be jointly convex on $(-1,1)\times\Real\times\Real_+$.

Indeed, consider two points $x_1 = (x_0', y_1), x_2 = (x_0', y_2)$ ($x_0'
\in \omega, -1 < y_1 < y_2 < 1$). We construct a function $u \in
\Wf(\Omega)$ with nonzero gradient only in circumferences of $x_1,
x_2$ and around the lateral boundary of a cylinder with axis
$[x_1, x_2]$. Namely,
$$
u(x) = \min\Big(
  \big( {y - y_1 \over c_1} + (x' - x_0') \cdot \B_1' \big)_+,
  \big( {y_2 - y \over c_2} + (x' - x_0') \cdot \B_2' \big)_+,
  \delta \big( w - |x' - x_0'| \big)_+, h
\Big).
$$
Here the parameters $c_1, c_2 > 0$ are inverse derivatives w.r.t. $y$ at the ``bases'' of the cylinder,
$\B_1', \B_2' \in \Real^{n - 1}$ are gradients w.r.t. $x'$ at the ``bases'',
$\delta > 0$ is the absolute value of the gradient at the lateral boundary, $w > 0$ is the cylinder radius,
while $h > 0$ is the maximal function value.

Given $y_1$, $y_2$, $c_1$, $c_2$, $\B_1'$, $\B_2'$, we choose $h$ and $\delta$ as functions of small parameter $w$.
We set $\varkappa\equiv {h \over \delta} := {w \over 2}$ ($\varkappa$ is the width of the lateral layer with non-zero gradient).

The left base of the support of $u$ is given by the system
$$
{y - y_1 \over c_1} + (x' - x_0') \cdot \B_1' = 0; \qquad |x' - x_0'| \le w.
$$
Thus, it is an $(n-1)$-dimensional prolate ellipsoid of revolution
with the major semi-axis $\sqrt{w^2 + c_1^2 w^2 |\B_1'|^2}$ and
radius $w$.
Therefore, $\nabla w = (\B_1', \frac{1}{c_1})$ on the set $A_1$ that is a truncated cone based on this ellipsoid.
Direct computation gives
$$
|A_1| = 
C_1 \delta c_1 w^n
$$
(hereinafter $C$ with or without indices are absolute constants).

Similarly, $\nabla w = (\B_2', -\frac{1}{c_2})$ on the set $A_2$ with $|A_2| = C_1 \delta c_2 w^n$.

After symmetrization the gradients of $u^*$ at the ``bases'' are equal
to $(\frac{c_1 \B_1' + c_2 \B_2'}{c_1 + c_2}, \pm \frac{2}{c_1 + c_2})$.
Hence the sets $A_1$ and $A_2$ become $A_1'$ and $A_2'$
with $|A_1'| = |A_2'| = C_1 \delta \frac{c_1 + c_2}{2} w^n$.

Next, denote by $A_\delta$ the lateral layer with non-zero
gradient. Direct estimate gives
$$
|A_\delta|
\le C\big((y_2 - y_1) w^{n - 1} + (c_1 |\B_1'| + c_2 w
|\B_2'|)w^n\big).
$$
Also denote
$$ Z = {\rm meas}\{ x \in \Omega \,|\, \nabla u(x) = 0 \} = {\rm
meas}\{ x \in \Omega \,|\, \nabla u^*(x) = 0 \}.
$$


Setting $\frac{1}{c_1} = \frac{1}{c_2} = w^2$, $\delta = w^4$ and
$\B_1' = \B_2' = 0$, the theorem assumptions imply
\begin{eqnarray*}
0 &\le& (\widehat\I(u) - Z ) - (\widehat\I(u^*) - Z )
\\ &\le& \big( (1 + w^4)^{\frac{p(\bar{x}_1)}{2}} - 1 \big) |A_1| + \big( (1 + w^4)^{\frac{p(\bar{x}_2)}{2}} - 1 \big) |A_2|
+ \big( (1 + w^8)^{\frac{P}{2}} - 1 \big) |A_\delta|
\\ &&- \big( (1 + w^4)^{\frac{p(\hat{x}_1)}{2}} - 1 + (1 + w^4)^{\frac{p(\hat{x}_2)}{2}} - 1 \big) |A_1'|
\\ &\le& w^4 \big({\textstyle\frac{p(\bar{x}_1)}{2} + \frac{p(\bar{x}_2)}{2} + o_w(1)}\big) C_1 w^{n + 2}
+ w^8 \big({\textstyle\frac{P}{2} + o_w(1)} \big) (y_2 - y_1 ) C
w^{n-1}
\\ &&- w^4 \big({\textstyle\frac{p(\hat{x}_1)}{2} + \frac{p(\hat{x}_2)}{2} + o_w(1)}\big) C_1 w^{n + 2}.
\end{eqnarray*}
Here $P = \max p(x', y)$, $\bar{x}_1 \in A_1$, $\bar{x}_2 \in
A_2$, $\hat{x}_1 \in A_1'$, $\hat{x}_2 \in A_2'$.

We push $w$ to $0$ and arrive at the inequality
$$
0 \le p(x_0',y_1) + p(x_0', y_2) - p(x_0', \frac{y_1 - y_2}{2}) -
p(x_0', \frac{y_2 - y_1}{2}).
$$
Applying Proposition
\ref{convProp} we prove that $p$ is convex and even in $y$.

Now choose arbitrary positive $c_1$, $c_2$, $d_1$ and $d_2$, take
$\B_1'=\frac {d_1}{c_1}\, {\bf e}$, $\B_2'=\frac {d_2}{c_2}\, {\bf
e}$ (here ${\bf e}$ is arbitrary unit vector in $x'$ hyperplane) and
set $\delta = w^2$. Then we obtain
\begin{eqnarray*}
0 &\le& (\widehat\I(u) - Z) - (\widehat\I(u^*) - Z)
\\ &\le& \big( \K_{\bar{x}_1'}(c_1, d_1, \bar{y}_1)-1\big) |A_1| + \big(\K_{\bar{x}_2'}(c_2, d_2, \bar{y}_2)-1\big) |A_2|
+ \big( (1 + w^4)^{\frac{P}{2}} - 1 \big) |A_\delta|
\\ &&-\big( {\textstyle\K_{\hat{x}_1'}(\frac{c_1 + c_2}{2}, \frac{d_1 + d_2}{2}, \hat{y}_1)-1 + \K_{\hat{x}_2'}(\frac{c_1 + c_2}{2}, \frac{d_1 + d_2}{2},
\hat{y}_2)-1 }\big) |A_1'|
\\ &\le& \big( {\textstyle \K_{\bar{x}_1'}(c_1, d_1, \bar{y}_1) + \K_{\bar{x}_2'}(c_2, d_2, \bar{y}_2)} \big) C_1 w^{n + 2}
+ ({\textstyle\frac{P}{2} + o_w(1)}) C (y_2 - y_1) w^{n + 3}
\\ &&-\big( {\textstyle\K_{\hat{x}_1'}(\frac{c_1 + c_2}{2}, \frac{d_1 + d_2}{2}, \hat{y}_1) +\K_{\hat{x}_2'}(\frac{c_1 +c_2}{2},
\frac{d_1 + d_2}{2}, \hat{y}_2)} \big) C_1 w^{n + 2}.
\end{eqnarray*}
Here $P = \max p(x', y)$, $(\bar{x}_1', \bar{y}_1)\in A_1$,
$(\bar{x}_2', \bar{y}_2) \in A_2$, $(\hat{x}_1', \hat{y}_1) \in
A_1'$, $(\hat{x}_2', \hat{y}_2) \in A_2'$.

Pushing $w$ to $0$, we obtain
$$
\K_{x_0'}(c_1, d_1, y_1) + \K_{x_0'}(c_2, d_2, y_2) \ge
\K_{x_0'}\big({\textstyle \frac{c_1 + c_2}{2}, \frac{d_1 + d_2}{2},\frac{y_1 - y_2}{2}
}\big) + \K_{x_0'}\big({\textstyle \frac{c_1 + c_2}{2}, \frac{d_1 + d_2}{2}, \frac{y_2 -y_1}{2}}\big).
$$
This implies $\K_{x_0'}$ is convex since it is even in $y$, and the claim follows.

Finally, for the sake of brevity, we omit $x'$ in the notation $\K_{x'}$ and notice that
$$
\textstyle
\K(c,d,y)=K(\frac c{\sqrt{1+d^2}},y)\cdot \sqrt{1+d^2},
$$
where the function $K$ is introduced in Theorem \ref{necessary}. Direct calculation gives
$$
\det(\K''(c, d, y))=\frac 1{(1+d^2)^{\frac 32}}\cdot\big[(\partial^2_{yy}K\partial^2_{ss}K-\partial^2_{sy}K^2)(K-s\partial_sK)-\partial^2_{ss}K\partial_yK^2d^2\big]
$$
(here $s=\frac c{\sqrt{1+d^2}}$). Therefore, if $\partial_y K \not\equiv 0$ we can choose $d$ sufficiently large so that  $\det(\K''(c, d, y))<0$, a contradiction.
\end{proof}

\section{Appendix}

To prove inequality (\ref{calc}) we split the positive quadrant
$(w,q) \in \Real_+ \times \Real_+$ into five regions, see Fig.~1:
\begin{eqnarray*}
 &&R_1 = [0, 6] \times [0, 1],\quad R_2 =[0, 1] \times [1,\infty],
\quad R_3 = [1, 4] \times [1, \infty],\\
&&R_4 = [6,\infty]\times [0,1], \quad
R_5=[4,\infty]\times[1,\infty].
\end{eqnarray*}
 On each of these regions we prove the inequality in numerical-analytical way.

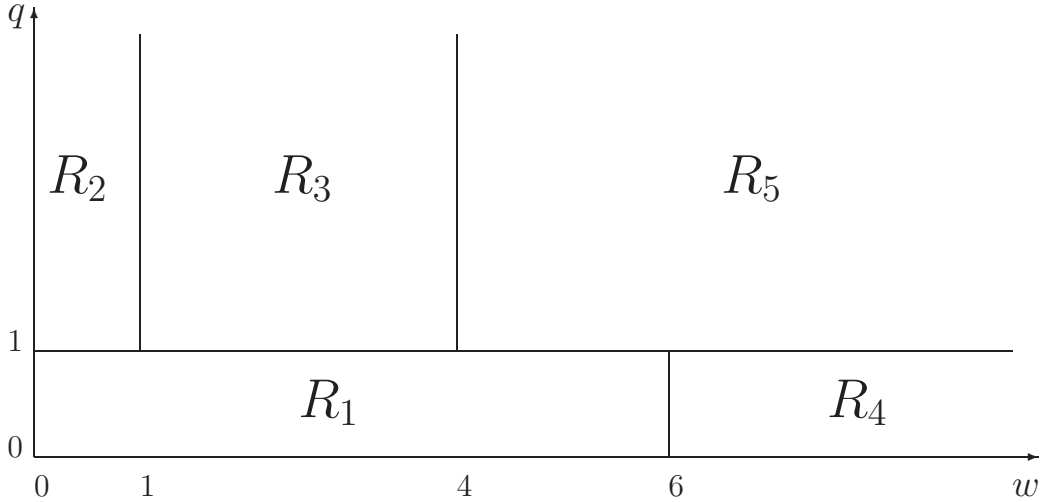
\begin{figure}[ht]
\begin{picture}(400,200)

\put(390,5){\large $w$} \put(20,5){$0$} \put(60,5){$1$}
\put(260,5){$6$} \put(180,5){$4$}

\put(10,20){$0$} \put(10,60){$1$} \put(10,185){\large $q$}

\put(20,20){\vector(1,0){380}} \put(20,60){\line(1,0){370}}

\put(20,20){\vector(0,1){170}} \put(60,60){\line(0,1){120}}
\put(260,20){\line(0,1){40}} \put(180,60){\line(0,1){120}}

\put(120,35){\LARGE $R_1$} 
\put(25,120){\LARGE $R_2$} \put(110,120){\LARGE $R_3$}
\put(320,35){\LARGE $R_4$}
\put(280,120){\LARGE $R_5$} 

\end{picture}

\caption{To the proof of (\ref{calc})}
\end{figure}


For $(w, q) \in R_1$ we construct appropriate piecewise constant
function $A_1(w,q)$ estimating $A(w,q)$ from above. To do this we
take a rectangle
$$
Q\equiv\{w_0\le w\le w_1;\quad q_0\le q\le q_1\}\subset R_1
$$
and obtain the constant value for $A_1$ on this rectangle
replacing arguments in the formula of $A(w,q)$ by their extremal
values in this rectangle:
\begin{multline*}
A(w,q) = \frac{ q ( 4 w - ( w + 3 ) \ln( w + 1 ) ) - (1 - {1 \over
w}) \ln( w + 1 ) + 4 {w \over \ln( w + 1 )} - 4 }{ 2 ( q w + 1 ) }
\cdot {q \over q + 1}
\\ \le \frac{
[4 q_1 w_1] - [q_0 (w_0 + 3) \ln(w_0 + 1)] - [\ln(w_0 + 1)] +
[{\ln(w_0 + 1) \over w_0}] + [4 {w_1 \over \ln(w_1 + 1)}] - 4 }{
[2 (q_0 w_0 + 1)] }
\\ \times [{q_1 \over q_1 + 1}]=:A_1|_Q.
\end{multline*}
Here we used the fact that every function in square brackets is
monotone on $R_1$ in both variables.

For $(w,q) \in R_2$ we put $r = {1 \over q}$ and construct
appropriate piecewise constant function $A_2(w,r)$ estimating
$A(w,{1 \over r})$ from above.

In a similar way, for $(w,q) \in R_4$ we put $v = {1 \over w}$ and
construct appropriate piecewise constant function $A_4(v,q)$
estimating $A({1 \over v},q)$ from above.

Finally, for $(w,q) \in R_5$ we put $v = {1 \over w}$, $r = {1
\over q}$ and construct appropriate piecewise constant function
$A_5(v,r)$ estimating $A({1 \over v},{1 \over r})$ from above.

The estimating functions $A_1$, $A_2$, $A_4$, $A_5$ were
calculated on proper meshes with $15$ valid digits. This gives the
following results.

\begin{center}
\begin{tabular} {|c|c|c|c|c|}
\hline Region & & Mesh step in $w(v)$ & Mesh step in $q(r)$ &
Inequality
\\
\hline $R_1$ & & $6\cdot10^{-2}$ & $10^{-1}$ & $A_1\le0.51$ \\
\hline $R_2$ & & $10^{-2}$ & $10^{-2}$  & $A_2\le0.617$ \\
\hline $R_4$ & & $2\cdot10^{-2}$ & $10^{-1}$  & $A_4\le0.50$ \\
\hline $R_5$ & & $2\cdot10^{-3}$ & $10^{-2}$  & $A_5\le0.605$ \\
\hline\end{tabular}
\end{center}
\medskip
In all cases we obtain $A(w,q)\le 0.62$.

The analysis of $A(w,q)$ in $R_3$ is more careful. We again put $r
= {1 \over q}$ and claim that $A(w,{1 \over r})$ is decreasing in
$r$. To prove this we construct appropriate piecewise constant
function $A_3(w,r)$ estimating $\partial_rA(w,{1 \over r})$ from
above. This function was calculated on proper mesh with $15$ valid
digits. This gives the following result.

\begin{center}
\begin{tabular} {|c|c|c|c|c|}
\hline Region & & Mesh step in $w$ & Mesh step in $r$ & Inequality
\\
\hline $R_3$ & & $5\cdot10^{-3}$ & $10^{-3}$ & $A_3\le-0.08$ \\
\hline\end{tabular}
\end{center}
\medskip
Thus, the claim follows, and $A(w,{1 \over r})$ attains its
maximum in $R_3$ at the line $r=0$.

To find the actual maximum we consider
$$
A(w, \infty) = 2 - {1 \over 2} ( \ln(w + 1) + 3 {\ln(w + 1) \over
w} )
$$
and claim that that $A(w, \infty)$ is concave for $w \in [1, 4]$.
To prove this we construct appropriate piecewise constant
function $A_\infty(w)$ estimating $A_\infty''(w)$ from above. This
function was calculated on proper mesh with $15$ valid digits.
This gives the following result.

\begin{center}
\begin{tabular} {|c|c|c|c|}
\hline Region & & Mesh step & Inequality
\\
\hline $1\le w\le 4$ & & $3 \cdot 10^{-3}$ & $A_\infty \le -0.13$ \\
\hline\end{tabular}
\end{center}
Thus, the maximum is unique. Using standard computational methods
we find that it is achieved for $w \approx 1.816960565240$, $\max
A(w,\infty)\approx 0.627178211634$, and (\ref{calc}) follows.

To prove inequality (\ref{calcHalf}) we split $(w, q) \in \Real_+
\times [0, 1.36]$ into four regions, see Fig.~2:
\begin{eqnarray*}
 &&R_6 = [0, 3] \times [0, 1.36],\quad R_7 = [3, 5] \times [0, 1.3], \\
 && R_8 = [3,5] \times [1.3, 1.36],\quad R_9 =[5, \infty] \times [0,1.36].
\end{eqnarray*}

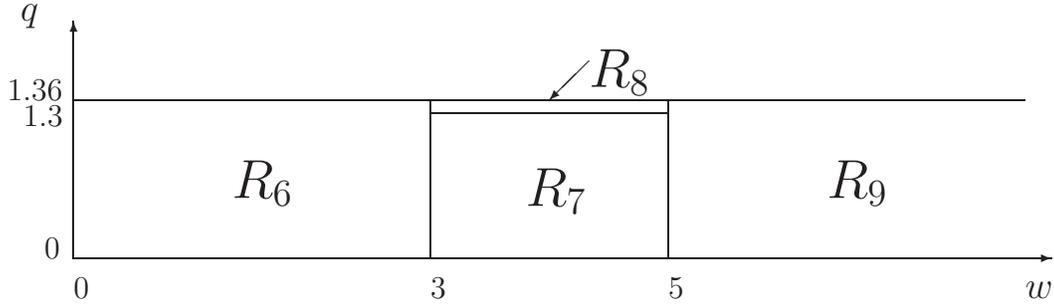
\begin{figure}[ht]
\begin{picture}(400,120)

\put(390,5){\large $w$} \put(30,5){$0$} \put(165,5){$3$} \put(255,5){$5$}

\put(19,20){$0$} \put(11,70){$1.3$} \put(5,80){$1.36$} \put(10,110){\large $q$}

\put(30,20){\vector(1,0){370}}
\put(30,80){\line(1,0){360}} \put(165,75){\line(1,0){90}}

\put(30,20){\vector(0,1){90}}
\put(165,20){\line(0,1){60}} \put(255,20){\line(0,1){60}}

\put(90,43){\LARGE $R_6$} \put(315,43){\LARGE $R_9$}
\put(201,40){\LARGE $R_7$} \put(225,85){\LARGE $R_8$}

\put(225,95){\vector(-1,-1){15}}

\end{picture}

\caption{To the proof of (\ref{calcHalf})}
\end{figure}

In these regions we used piecewise constant functions $A_1$ and
$A_4$ introduced earlier. These functions were calculated on
proper meshes with $15$ valid digits.
In $R_8$ we need to use the mesh step less then
$10^{-5}$, so we repeated calculations with $18$ valid digits.
 This gives the following results.

\begin{center}
\begin{tabular} {|c|c|c|c|c|}
\hline Region & & Mesh step in $w(v)$ & Mesh step in $q$ &
Inequality
\\
\hline $R_6$ & & $3 \cdot 10^{-3}$ & $1.36 \cdot 10^{-3}$ & $A_1 \le 0.498$ \\
\hline $R_7$ & & $2 \cdot 10^{-3}$ & $1.3 \cdot 10^{-3}$  & $A_1 \le 0.498$ \\
\hline $R_8$ & & $2 \cdot 10^{-4}$ & $6 \cdot 10^{-6}$  & $A_1 \le 0.49996$ \\
\hline $R_9$ & & $2 \cdot 10^{-3}$ & $1.36 \cdot 10^{-2}$  & $A_4 \le 0.4992$ \\
\hline\end{tabular}
\end{center}

The proof is complete.

\section*{Acknowledgements}

Authors are grateful to M.~Surnachev for simplification of the proof of Theorem \ref{mainThm}.



\begin{thebibliography}{99}

\bibitem{B-ZNS} S.~V.~Bankevich, ``On monotonicity of some functionals under monotone rearrangement with respect to one variable'',
ZNS POMI, {\bf 444} (2016), 5--14 (Russian); English transl.: J. Math. Sci. {\bf 224} (2017), N3, 385--390.

\bibitem{B-FAA} S.~V.~Bankevich, ``On the P\'olya--Szeg\"o inequality for functionals with variable exponent of summability'',
Funk. an. i prilozh. {\bf 52} (2018), to appear (Russian).

\bibitem{BN-DAN} S.~V.~Bankevich, A.~I.~Nazarov, ``On the P\'olya--Szeg\"o inequality generalization for one-dimensional
functionals'', Doklady RAN, {\bf 438} (2011), N1, 11--13 (Russian); English transl.: Doklady Math. {\bf 83}
(2011), N3, 287--289.

\bibitem{BN-CV} S.~V.~Bankevich, A.~I.~Nazarov, ``On monotonicity of some functionals under rearrangements'',
Calc. Var. and PDEs~{\bf53} (2015), 627--647.

\bibitem{Br} F.~Brock, ``Weighted Dirichlet-type inequalities for Steiner symmetrization'',
Calc. Var. and PDEs~{\bf8} (1999), 15--25.

\bibitem{BGH} G.~Buttazzo, M.~Giaquinta, S.~Hildebrandt, ``One-Dimentional Variational Problems. An Introduction'',
Oxford Lecture Series in Mathematics and Its Applications~{\bf15}, Oxford University Press, New York (1998).

\bibitem{DHHR} L.~Diening, P.~Harjulehto, P.~H\"ast\"o, M.~Ru\v{z}i\v{c}ka, ``Lebesgue and Sobolev spaces with
Variable Exponents'', Lecture Notes in Mathematics~{\bf2017}, Springer, Berlin (2011).

\bibitem{Shar} I.~I.~Sharapudinov, ``On certain problems of approximation theory in the Lebesgue spaces with variable exponent'',
Mathematics monographs series~{\bf5}, Vladikavkaz (2012) (Russian).

\bibitem{Tal} G.~Talenti, ``The art of rearranging'', Milan Journ. of Math.~{\bf84} (2016), N1, 105--157.

\bibitem{Zh} V.~V.~Zhikov, ``On variational problems and nonlinear elliptic equations with nonstandard growth conditions'',
Probl. Math. An.~{\bf54} (2011), 23--112 (Russian); English transl.: J. Math. Sci. {\bf173} (2011) N5, 463--570.

\end{thebibliography}
\end{document}